\documentclass[11pt]{amsart}

\usepackage{etex}
\usepackage{amsmath,calligra,mathrsfs}
\usepackage{amssymb}
\usepackage{amscd}
\usepackage{amsthm}
\usepackage{hyperref}
\usepackage{colortbl}
\usepackage{color}
\usepackage{mathtools}
\usepackage{extarrows}
\usepackage{ifthen}
\usepackage{amsmath}

\usepackage[all,cmtip]{xy}

\usepackage{tikz}
\usepackage{pgfplots}
\usepackage{tikz-cd}

\newtheorem*{thm*}{Theorem 1}
\newtheorem*{thm**}{Theorem 2}
\newtheorem*{thm***}{Theorem 3}

\newtheorem{thm}{Theorem}[section]
\newtheorem{lem}[thm]{Lemma}

\newtheorem{conj}[thm]{Conjecture}
\newtheorem{prop}[thm]{Proposition}
\newtheorem{question}[thm]{Question}

\theoremstyle{remark}

\newtheorem{rem}[thm]{Remark}

\theoremstyle{definition}
\newtheorem{defn}[thm]{Definition}

\numberwithin{equation}{section}

\allowdisplaybreaks[4]

\begin{document}

\vfuzz0.5pc
\hfuzz0.5pc 

\newcommand{\claimref}[1]{Claim \ref{#1}}
\newcommand{\thmref}[1]{Theorem \ref{#1}}
\newcommand{\propref}[1]{Proposition \ref{#1}}
\newcommand{\lemref}[1]{Lemma \ref{#1}}
\newcommand{\coref}[1]{Corollary \ref{#1}}
\newcommand{\remref}[1]{Remark \ref{#1}}
\newcommand{\conjref}[1]{Conjecture \ref{#1}}
\newcommand{\questionref}[1]{Question \ref{#1}}
\newcommand{\defnref}[1]{Definition \ref{#1}}
\newcommand{\secref}[1]{\S \ref{#1}}
\newcommand{\ssecref}[1]{\ref{#1}}
\newcommand{\sssecref}[1]{\ref{#1}}

\newcommand{\RED}{{\mathrm{red}}}
\newcommand{\tors}{{\mathrm{tors}}}
\newcommand{\eq}{\Leftrightarrow}

\newcommand{\mapright}[1]{\smash{\mathop{\longrightarrow}\limits^{#1}}}
\newcommand{\mapleft}[1]{\smash{\mathop{\longleftarrow}\limits^{#1}}}
\newcommand{\mapdown}[1]{\Big\downarrow\rlap{$\vcenter{\hbox{$\scriptstyle#1$}}$}}
\newcommand{\smapdown}[1]{\downarrow\rlap{$\vcenter{\hbox{$\scriptstyle#1$}}$}}

\newcommand{\A}{{\mathbb A}}
\newcommand{\I}{{\mathcal I}}
\newcommand{\J}{{\mathcal J}}
\newcommand{\CO}{{\mathcal O}}
\newcommand{\C}{{\mathscr{C}}}
\newcommand{\CC}{{\mathcal C}}
\newcommand{\BC}{{\mathbb C}}
\newcommand{\BQ}{{\mathbb Q}}
\newcommand{\m}{{\mathcal M}}
\newcommand{\h}{{\mathcal H}}
\newcommand{\Z}{{\mathcal Z}}
\newcommand{\BZ}{{\mathbb Z}}
\newcommand{\W}{{\mathcal W}}
\newcommand{\Y}{{\mathcal Y}}
\newcommand{\T}{{\mathcal T}}
\newcommand{\BP}{{\mathbb P}}
\newcommand{\CP}{{\mathcal P}}
\newcommand{\G}{{\mathbb G}}
\newcommand{\g}{\mathscr{G}}
\newcommand{\CG}{{\mathcal G}}
\newcommand{\BR}{{\mathbb R}}
\newcommand{\D}{{\mathscr{D}}}
\newcommand{\DD}{{\mathcal D}}
\newcommand{\LL}{{\mathcal L}}
\newcommand{\f}{\mathscr{F}}
\newcommand{\CF}{{\mathcal F}}
\newcommand{\E}{{\mathcal E}}
\newcommand{\e}{{\mathscr{E}}}
\newcommand{\BN}{{\mathbb N}}
\newcommand{\N}{{\mathcal N}}
\newcommand{\n}{{\mathscr{N}}}
\newcommand{\K}{{\mathcal K}}
\newcommand{\R} {{\mathbb R}}
\newcommand{\PP}{{\mathbb P}}
\newcommand{\Pp}{{\mathbb P}}
\newcommand{\BF}{{\mathbb F}}
\newcommand{\QQ}{{\mathcal Q}}
\newcommand{\closure}[1]{\overline{#1}}
\newcommand{\EQ}{\Leftrightarrow}
\newcommand{\imply}{\Rightarrow}
\newcommand{\isom}{\cong}
\newcommand{\embed}{\hookrightarrow}
\newcommand{\tensor}{\mathop{\otimes}}
\newcommand{\wt}[1]{{\widetilde{#1}}}
\newcommand{\ol}{\overline}
\newcommand{\ul}{\underline}

\newcommand{\bs}{{\backslash}}
\newcommand{\CS}{{\mathcal S}}
\newcommand{\CA}{{\mathcal A}}
\newcommand{\Q}{{\mathbb Q}}
\newcommand{\F}{{\mathcal F}}
\newcommand{\sing}{{\text{sing}}}
\newcommand{\U} {{\mathcal U}}
\newcommand{\B}{{\mathcal B}}
\newcommand{\X}{{\mathcal X}}

\newcommand{\ECS}[1]{E_{#1}(X)}
\newcommand{\CV}[2]{{\mathcal C}_{#1,#2}(X)}

\newcommand{\rank}{\mathop{\mathrm{rank}}\nolimits}
\newcommand{\codim}{\mathop{\mathrm{codim}}\nolimits}
\newcommand{\Ord}{\mathop{\mathrm{Ord}}\nolimits}
\newcommand{\Var}{\mathop{\mathrm{Var}}\nolimits}
\newcommand{\Ext}{\mathop{\mathrm{Ext}}\nolimits}
\newcommand{\EXT}{\mathop{\mathscr{E}{\kern -2pt {xt}}}\nolimits}
\newcommand{\Pic}{\mathop{\mathrm{Pic}}\nolimits}
\newcommand{\Spec}{\mathop{\mathrm{Spec}}\nolimits}
\newcommand{\Jac}{\mathop{\mathrm{Jac}}\nolimits}
\newcommand{\Div}{\mathop{\mathrm{Div}}\nolimits}
\newcommand{\sgn}{\mathop{\mathrm{sgn}}\nolimits}
\newcommand{\supp}{\mathop{\mathrm{supp}}\nolimits}
\newcommand{\Hom}{\mathop{\mathrm{Hom}}\nolimits}
\newcommand{\HOM}{\mathop{\mathscr{H}{\kern -3pt {om}}}\nolimits}
\newcommand{\Sym}{\mathop{\mathrm{Sym}}\nolimits}
\newcommand{\nilrad}{\mathop{\mathrm{nilrad}}\nolimits}
\newcommand{\Ann}{\mathop{\mathrm{Ann}}\nolimits}
\newcommand{\Proj}{\mathop{\mathrm{Proj}}\nolimits}
\newcommand{\mult}{\mathop{\mathrm{mult}}\nolimits}
\newcommand{\Bs}{\mathop{\mathrm{Bs}}\nolimits}
\newcommand{\Span}{\mathop{\mathrm{Span}}\nolimits}
\newcommand{\IM}{\mathop{\mathrm{Im}}\nolimits}
\newcommand{\Hol}{\mathop{\mathrm{Hol}}\nolimits}
\newcommand{\End}{\mathop{\mathrm{End}}\nolimits}
\newcommand{\CH}{\mathop{\mathrm{CH}}\nolimits}
\newcommand{\Exec}{\mathop{\mathrm{Exec}}\nolimits}
\newcommand{\SPAN}{\mathop{\mathrm{span}}\nolimits}
\newcommand{\birat}{\mathop{\mathrm{birat}}\nolimits}
\newcommand{\cl}{\mathop{\mathrm{cl}}\nolimits}
\newcommand{\rat}{\mathop{\mathrm{rat}}\nolimits}
\newcommand{\Bir}{\mathop{\mathrm{Bir}}\nolimits}
\newcommand{\Rat}{\mathop{\mathrm{Rat}}\nolimits}
\newcommand{\aut}{\mathop{\mathrm{aut}}\nolimits}
\newcommand{\Aut}{\mathop{\mathrm{Aut}}\nolimits}
\newcommand{\eff}{\mathop{\mathrm{eff}}\nolimits}
\newcommand{\nef}{\mathop{\mathrm{nef}}\nolimits}
\newcommand{\amp}{\mathop{\mathrm{amp}}\nolimits}
\newcommand{\DIV}{\mathop{\mathrm{Div}}\nolimits}
\newcommand{\Bl}{\mathop{\mathrm{Bl}}\nolimits}
\newcommand{\Cox}{\mathop{\mathrm{Cox}}\nolimits}
\newcommand{\NE}{\mathop{\mathrm{NE}}\nolimits}
\newcommand{\NM}{\mathop{\mathrm{NM}}\nolimits}
\newcommand{\Gal}{\mathop{\mathrm{Gal}}\nolimits}
\newcommand{\coker}{\mathop{\mathrm{coker}}\nolimits}
\newcommand{\ch}{\mathop{\mathrm{ch}}\nolimits}
\newcommand{\Gr}{\mathop{\mathrm{Gr}}\nolimits}
\newcommand{\Td}{\mathop{\mathrm{Td}}\nolimits}
\newcommand{\kodim}{\mathop{\mathrm{kodim}}\nolimits}
\newcommand{\Mov}{\mathop{\mathrm{Mov}}}
\newcommand{\Conv}{\mathop{\mathrm{Conv}}}
\newcommand{\Null}{\mathop{\mathrm{Null}}}

\newcommand{\comment}[1]{}

\ifthenelse{\equal{1}{1}}{
	\newcommand{\blue}[1]{{\color{blue}#1}}
	\newcommand{\green}[1]{{\color{green}#1}}
	\newcommand{\red}[1]{{\color{red}#1}}
}{
	\newcommand{\blue}[1]{}
	\newcommand{\green}[1]{}
	\newcommand{\red}[1]{}
}

\usetikzlibrary{decorations.markings}
\tikzset{degil/.style={
            decoration={markings,
            mark= at position 0.5 with {
                  \node[transform shape] (tempnode) {\tiny $/$};
                  }
              },
              postaction={decorate}
}
}

\title[Zariski's Conjecture and Euler-Chow Series]{Zariski's Conjecture and Euler-Chow Series}

\author{Xi Chen}
\address{Department of Mathematics and Statistics\\
	University of Alberta\\
	Edmonton, Alberta T6G 2G1, CANADA}
\email{xichen@math.ualberta.ca}
\thanks{Research of Chen was supported in part by NSERC 262265.}

\author{E. Javier Elizondo}
\address{Instituto de Matem\'aticas\\
Universidad Nacional Aut\'onoma de M\'exico\\Ciudad Universitaria\\
M\'exico DF 04510, M\'exico}
\email{javier@im.unam.mx}
\thanks{Research of Elizondo was supported in part by DGAPA 109515.}

\date{March 1, 2020}

\maketitle

\begin{abstract}
We study the relations between the finite generation of Cox ring, the rationality of Euler-Chow series and Poincar\'e series and Zariski's conjecture on dimensions of linear systems. We prove that if the Cox ring of a smooth projective variety is finitely generated, then all Poincar\'e series of the variety are rational. We also prove that the multi-variable Poincar\'e series associated to big divisors on a smooth projective surface are rational, assuming the rationality of multi-variable Poincar\'e series on curves.  
\end{abstract}

\section{Introduction}

Let $X$ be a smooth projective variety.
For an effective divisor $D$ on $X$, consider the subring
\begin{equation*}\label{E500}
R(X,D) = \bigoplus_{n=0}^\infty H^0(\CO_X(nD))
\end{equation*}
of the Cox ring $\Cox(X)$ of $X$ \cite{CoxRing} with associated formal power series
\begin{equation*}\label{E502}
R_{X,D}(t) = \sum_{n=0}^\infty h^0(\CO_X(nD)) t^n
\end{equation*}
in $\BZ[[t]]$. We call $R_{X,D}$ the {\em Poincar\'e series} associated to $D$.

We may think of $R_{X,D}$ as a sub-series of the {\em Euler-Chow series} of $X$ in codimension one, roughly given by
\begin{equation*}\label{GSRING2E000}
E_X = \sum_{D\in \Pic(X)} h^0(\CO_X(D)) t^D.
\end{equation*}
For the precise definition of $E_X$, we refer the readers to \cite{ElizondoEulerSeries} and \cite[Definition 2.1]{CEY}.
Also we give some known examples of $E_X$ in Appendix \ref{GSRING2APPENDIXA}.

The rationality of $E_X$ and its relation to the finite generation of $\Cox(X)$ were studied in \cite{KKT} and \cite{CEY}. Here we ask some further questions on $\Cox(X)$, $E_X$, $R(X,D)$ and $R_{X,D}$:

\begin{question}\label{Q000}
Let $X$ be a smooth projective variety whose Picard group $\Pic(X)$ is a free abelian group
of finite rank. Are these statements true:
\begin{enumerate}
\item
$\Cox(X)$ is finitely generated if and only if $E_X$ is rational?
\item
$\Cox(X)$ is finitely generated if and only if $R(X,D)$ is finitely generated
for every effective divisor $D$?
\item
$R(X,D)$ is finitely generated if and only if $R_{X,D}$ is rational?
\item
$E_X$ is rational if and only if $R_{X, D}$ is rational for every effective divisor $D$? 
\end{enumerate}
\end{question}
The answers to some of these questions are known: affirmative or negative; some are still wide
open. They are illustrated by the following diagram:
\begin{equation*}\label{E501}
\begin{tikzcd}
\text{$\Cox(X)$ f.g.}
\ar[r, Rightarrow, shift left=1.5]
\ar[r, Leftarrow, degil, shift right=1.5]
\ar[d, Rightarrow, shift left=1.5]
\ar[d, Leftarrow, degil, shift right=1.5]
&
\text{$E_X$ rational}
\ar[d, Rightarrow, shift left=1.5, "?"]
\ar[d, Leftarrow, degil, shift right=1.5]
\\
\text{$R(X,D)$ f.g. $\forall D$}
\ar[r, Rightarrow, shift left=1.5]
\ar[r, Leftarrow, degil, shift right=1.5]
&
\text{$R_{X,D}$ rational $\forall D$}
\end{tikzcd}
\end{equation*}


\begin{rem}\label{REM000000}
Here are a few comments:
\begin{enumerate}
\item
There are surfaces $X$ with $E_X$ rational and $\Cox(X)$ not finitely generated
\cite{CEY}. Therefore, 
\begin{equation*}\label{E503}
E_X \text{ rational } \not\Rightarrow \Cox(X) \text{ f.g.}
\end{equation*}
\item
A subring of a noetherian ring is not necessarily noetherian. Even if $\Cox(X)$ is a finitely
generated $\BC$-algebra, its sub-algebras are not necessarily finitely generated. However,
we still have
\begin{equation}\label{E504}
\Cox(X) \text{ f.g.} \Rightarrow R(X,D) \text{ f.g. for all } D
\end{equation}
by \cite[Proposition 1.11]{H-K} (see also \cite[Proposition 9.6]{OkawaImageMDS}).
On the other hand, taking $X$ to be the blowup of $\PP^2$ at $9$ points in general position,
one can easily show that $R(X,D)$ is finitely generated
for every effective divisor $D$ but $\Cox(X)$ is not finitely generated.
Therefore,
\begin{equation*}\label{E505}
R(X,D) \text{ f.g. for all } D \not\Rightarrow \Cox(X) \text{ f.g.}
\end{equation*}
\item
For a smooth projective surface $X$, $R_{X,D}$ is rational for every effective divisor $D$ by
Zariski's conjecture \cite{ZariskiConjecture} (see below).
On the other hand, Zariski found examples of $D$ such that $R(X,D)$
is not finitely generated \cite{ZariskiConjecture}. Therefore, 
\begin{equation*}\label{E507}
R_{X,D} \text{ rational } \not\Rightarrow R(X,D) \text{ f.g.}
\end{equation*}
\item
Finally, we expect it to be true that
\begin{equation*}\label{E508}
E_X \text{ rational } \Rightarrow R_{X,D} \text{ rational for all } D \text{ effective}; 
\end{equation*}
yet it is by no means obvious to us. On the other hand,
\begin{equation*}\label{E506}
R_{X,D} \text{ rational for all } D \not\Rightarrow E_X \text{ rational}
\end{equation*}
by taking $X$ to be the blowup of
$\PP^2$ at $9$ general points \cite{KKT} and \cite{CEY} (see below).
\end{enumerate}
\end{rem}

Our first result is

\begin{thm}\label{THM001}
For a Mori dream space (MDS) $X$ over $\BC$,
\begin{equation*}\label{GSRING2E044}
R(X,D_1,D_2,...,D_l) = \bigoplus_{m_i\in \BN} H^0(\CO_X(\sum_{i=1}^l m_iD_i))
\end{equation*}
is a finitely generated $\BC$-algebra and
\begin{equation*}\label{GSRING2E045}
M(X,D,D_1,D_2,...,D_l) = \bigoplus_{m_i\in \BN} H^0(\CO_X(D + \sum_{i=1}^l m_iD_i))
\end{equation*}
is a finitely generated module over $R(X,D_1,D_2,...,D_l)$
for all $\BQ$-effective divisors $D_1, D_2, ..., D_l\in \Pic(X)$, all $l\in \BZ^+$
and all $D\in \Pic(X)$.
\end{thm}

Y. Hu and S. Keel proved that for a smooth projective variety $X$ whose $\Pic(X)$ is a free abelian group of finite rank, $\Cox(X)$ is finitely generated if and only if $X$ is a MDS \cite{H-K}. So the above theorem implies \eqref{E504}.
See \S \ref{GSRING2SEC2} for the definition of MDS.

It should be pointed out that the above theorem is known to the experts, although the exact statement, to the best of our knowledge, did not appear in the literature. There are similar results in \cite{H-K}, \cite{CasciniLazicMMP1} and \cite{OkawaImageMDS}.
However, these results do not directly apply here. For example, in \cite[Lemma 2.20]{OkawaImageMDS}, it was proved that the ring
$$
R(X, \Gamma) = \bigoplus_{D\in \Gamma} H^0(D)
$$
is finitely generated, where $\Gamma$ is the subgroup of $\DIV(X)$ generated by $D_i$ for $i=1,2,...,l$. We have an obvious ring homomorphism
$$
\begin{tikzcd}
R(X, D_1, D_2, ..., D_l) \ar{r}{\phi} & R(X, \Gamma)
\end{tikzcd}
$$
But $\phi$ is neither injective nor surjective in general. It is not surjective because there might be effective divisors $D\in \Gamma$ such that
$D$ cannot be written as $D = \sum m_i D_i$ for some $m_i\in \BN$. Therefore, we cannot derive the finite generation of $R(X, D_1, D_2, ..., D_l)$ from that of $R(X, \Gamma)$.

As another note, in \cite[Theorem 2.13]{H-K}, it was proved that the cone generated by $D_1,D_2,...,D_l$ has a set of generators $D_1', D_2', ..., D_m'$ such that
$R(X, D_1', D_2', ..., D_m')$ is finitely generated. Since $\{D_i\}$ and $\{D_j'\}$ might be different, the rings $R(X,D_1,D_2,...,D_l)$ and $R(X,D_1',D_2',...,D_m')$ are not necessarily the same; finite generation of one of them does not imply that of the other.

On surfaces, rationality of Poincar\'e series is closely related to
Zariski's conjecture on linear systems, proved by Cutkosky and Srinivas in \cite{ZariskiConjecture}:

\begin{thm*}[Zariski's Conjecture by Cutkosky-Srinivas]\label{ZC}
Let $X$ be a smooth projective surface over $\BC$ and $D$ be a divisor on $X$. Then
there exist periodic functions $a(n)$, $b(n)$ and $c(n)$ such that
\begin{equation}\label{E511}
h^0(nD) = a(n) n^2 + b(n) n + c(n)
\end{equation}
for $n$ sufficiently large. In addition, $a(n)$ and $b(n)$ are constants if $D$ is effective.
\end{thm*}

It follows from \eqref{E511} that $R_{X,D}$ is rational for every divisor $D$ on a smooth projective
surface $X$. Indeed, \eqref{E511} can be interpreted as
\begin{equation*}\label{E512}
\delta_r^3 (h^0(nD)) = 0
\end{equation*}
for $n$ sufficiently large, where $r\in \BZ^+$ is a constant and $\delta_r$ is the difference operator
defined by $\delta_r (f(n)) = f(n) - f(n - r)$. Note that
\begin{equation*}\label{E513}
(1-t^r) R_{X,D}(t) = \sum \delta_r(h^0(nD)) t^n.
\end{equation*}
Therefore,
\begin{equation}\label{E514}
R_{X,D}(t) = \frac{f(t)}{(1-t^r)^3}
\end{equation}
for some $f(t)\in \BZ[t]$. Furthermore, if $D$ is effective, we have
\begin{equation*}\label{E515}
\delta_1^2 \delta_r (h^0(nD)) = 0
\end{equation*}
and hence
\begin{equation}\label{E516}
R_{X,D}(t) = \frac{f(t)}{(1-t)^2 (1-t^r)}
\end{equation}
for some $f(t)\in \BZ[t]$. It is basically the same discussion on \cite[p.535]{ZariskiConjecture}. Zariski's conjecture
can be put in the following equivalent form:

\begin{thm**}[Zariski's conjecture in terms of Poincar\'e series]\label{ZC2}
Let $X$ be a smooth projective surface over $\BC$ and $D$ be a divisor on $X$. Then $R_{X,D}$ is given by \eqref{E514} for arbitrary $D$ and by \eqref{E516} for effective $D$.
\end{thm**}

We want to generalize it to the Poincar\'e series
associated to multiple divisors. Let
\begin{equation*}\label{E517}
R_{X,D_1,D_2,...,D_l}(t_1, t_2, ..., t_l) = \sum_{m_i\in\BN} h^0(\sum_{i=1}^l m_i D_i) t_1^{m_1} t_2^{m_2}
... t_l^{m_l}
\end{equation*}
to be the (multi-variable) Poincar\'e series associated to the divisors $D_i$, as a formal series in $\BZ[[t_1,t_2,...,t_l]]$. For induction purpose, we would also introduce an additional divisor $D$ and define
\begin{equation*}\label{GSRING2E001}
M_{X,D,D_1,D_2,...,D_l}(t_1, t_2, ..., t_l)
= \sum_{m_i\in\BN} h^0(D+\sum_{i=1}^l m_i D_i) t_1^{m_1} t_2^{m_2}
... t_l^{m_l}.
\end{equation*}

As a generalization of Cutkosky-Srinivas' theorem,
we expect the following to be true:

\begin{conj}\label{CONJ000}
For all pseudo-effective $\BQ$-divisors
$D, D_1, D_2, ..., D_l$ on a smooth projective surface $X$ over $\BC$,
\begin{equation}\label{E518}
M_{X,D,D_1,D_2,...,D_l}(t_1, t_2, ..., t_l) = \frac{f(t_1, t_2, ..., t_l)}{g(t_1, t_2, ..., t_l)}
\end{equation}
is rational for some $f, g\in \BZ[t_1, t_2, ..., t_l]$ satisfying $g(0,0,...,0) = 1$.
\end{conj}

Of course, the above conjecture holds when $X$ is a MDS by Theorem \ref{THM001}.
The main reason we study this conjecture is that we need it for the following:

\begin{conj}\label{CONJ001}
Let $X$ be a smooth projective surface over $\BC$ whose $\Pic(X)$ is a free abelian group
of finite rank. Then $E_X$ is rational if and only if the closed cone $\NE^1(X)\subset H^2(X,\BR)$ of pseudo-effective divisors on $X$ is rational polyhedral.
\end{conj}

The ``only if'' part was proved in \cite{KKT}. Actually, we have
\begin{equation*}\label{GSRING2E002}
E_X \text{ rational } \Rightarrow \NE^1(X) \text{ rational polyhedral}
\end{equation*}
in all dimensions.

\begin{proof}[Proof that \conjref{CONJ000} $\Rightarrow$ \conjref{CONJ001}]
Assuming that $\NE^1(X)$ is rational polyhedral, we can divide it into a union of rational simplicial cones
\begin{equation*}\label{E519}
\NE^1(X) = S_1 \cup S_2 \cup ... \cup S_n
\end{equation*}
such that the intersections
\begin{equation*}\label{GSRING2E003}
S_I = \bigcap_{i\in I} S_i
\end{equation*}
are also rational simplicial for all $I\subset \{1,2,...,n\}$. Then
\begin{equation*}\label{GSRING2E004}
E_X = \sum_{\substack{I\subset \{1,2,...,n\}\\ I\ne \emptyset}} (-1)^{|I|+1} \sum_{D\in S_I\cap H^2(X,\BZ)}
h^0(D) t^D.
\end{equation*}
So it suffices to show that the series
\begin{equation*}\label{GSRING2E005}
\sum_{D\in S\cap H^2(X,\BZ)} h^0(D) t^D
\end{equation*}
is rational for every rational simplicial cone $S\subset \NE^1(X)$.

Suppose that $S$ is spanned by
\begin{equation*}\label{GSRING2E006}
S = \big\{ \lambda_1 F_1 + \lambda_2 F_2 + ... + \lambda_a F_a:
\lambda_i \ge 0\big\}
\end{equation*}
where $F_i$ are pseudo-effective divisors on $X$ spanning the extremal rays of $S$. Since $S$ is a simplicial cone, every divisor
$D\in S\cap H^2(X, \BZ)$ can be written in a unique way as
\begin{equation*}\label{E520}
D = E + m_1 F_1 + m_2 F_2 + ... + m_a F_a
\end{equation*}
for some 
\begin{equation*}\label{GSRING2E007}
E\in \Lambda = \{ \lambda_1 F_1 + \lambda_2 F_2 + ... + \lambda_a F_a:
0\le \lambda_i < 1
\}\cap H^2(X,\BZ)
\end{equation*}
and $m_1, m_2, ..., m_a\in \BN$.
Clearly, $\Lambda$ is a finite set. Then
\begin{equation*}\label{E521}
\begin{aligned}
\sum_{D\in S\cap H^2(X,\BZ)} h^0(D) t^D
&= \sum_{E\in \Lambda} \sum_{m_i\in \BN}
h^0(E+\sum m_i F_i) t^{E+\sum m_i F_i}
\\
&=\sum_{E\in \Lambda} t^E M_{X,E,F_1,F_2,...,F_a}(t^{F_1},t^{F_2},...,t^{F_a})
\end{aligned}
\end{equation*}
is rational by \conjref{CONJ000}.
\end{proof}

Unfortunately, we are not able to prove Conjecture \ref{CONJ000} for all $\BQ$-effective divisors $D_i$. At the moment, we have a partial generalization of Zariski's conjecture, assuming the rationality of Poincar\'e series of curves: 

\begin{conj}\label{GSRING2CONJ000}
Let $X$ be a smooth projective surface over $\BC$
and let $C$ be a closed subscheme of $X$ of pure dimension $1$, supported on 
a divisor of simple normal crossings (snc) on $X$. Let $\LL_1, \LL_2,...,\LL_a$ and $\D$ be line bundles on $C$ with $\LL_i$ numerically trivial. Then
the formal power series
\begin{equation*}\label{GSRING2E020}
\sum_{m_1,m_2,...,m_a\in \BN} h^q(\D\otimes \LL_1^{\otimes m_1} \otimes \LL_2^{\otimes m_2} \otimes ... \otimes \LL_a^{\otimes m_a})
t_1^{m_1} t_2^{m_2} ... t_a^{m_a} 
\end{equation*}
in $\BZ[[t_1,t_2,...,t_a]]$ is rational, i.e., lies in $\BQ(t_1,t_2,...,t_a)$, for all $q\in \BN$.
\end{conj}

This conjecture is closely related to Mordell-Lang, formulated by Serge Lang
\cite[Chap. 8, Sec. 8, p. 221]{LangDiophantine}. 
In the case that $C$ is smooth, it is implied by Mordell-Lang.
A further study of \ref{GSRING2CONJ000} is planned in a future paper.
Assuming it, we can prove the following:

\begin{thm}\label{GSRING2THM000}
Assuming Conjecture \ref{GSRING2CONJ000},
\eqref{E518} holds for all big divisors $D_1, D_2, ..., D_l$ and all divisors $D$ on a smooth projective surface $X$ over $\BC$.
\end{thm}

The paper is organized as follows: we will prove Theorem \ref{THM001} in \S \ref{GSRING2SEC2} and Theorem \ref{GSRING2THM000} in \S \ref{GSRING2SEC003}.
We work exclusively over the complex numbers.

\section{Sub Cox Rings of Mori Dream Spaces}\label{GSRING2SEC2}

In this section, we prove Theorem \ref{THM001}.
Let us first fix some notations and go through some basic concepts.

Let $Z^k(X)$ be the free abelian group of codimension $k$ cycles on $X$. A $\BQ$-divisor ($\BR$-divisor) is a vector in $Z_\BQ^1(X) = Z^1(X)\otimes \BQ$ (respectively, $Z_\BR^1(X) = Z^1(X)\otimes \BR$). A divisor/$\BQ$-divisor/$\BR$-divisor $D = \sum x_i D_i$ is effective if $x_i\ge 0$ for all $x_i$, where $D_i$ are integral hypersurfaces of $X$. A divisor/$\BQ$-divisor/$\BR$-divisor $D$ is nef if $D\Gamma \ge 0$ for all irreducible curves $\Gamma\subset X$.
We say a divisor $D$ is {\em movable} if $|D|$ has no fixed part and 
a $\BQ$-divisor $D$ is {\em $\BQ$-movable} if $a D$ is movable for some positive integer $a$.

For $D = \sum_{i=1}^m x_i D_i\in Z_\BR^k(X)$ with $x_i\ne 0$ and $D_i$ reduced and irreducible, the support of $D$ is $\supp(D) = \sum_{i=1}^m D_i$.

Considering the image of $Z^k(X)$ in $H^{2k}(X, \BR)$, we
let $\Mov(X)$ and $\NM^1(X)$ be the smallest closed cones in $H^2(X, \BR)$ containing the images of
movable and nef divisors on $X$,
respectively, and let $\NE^k(X) = \NE_{\dim X - k} (X)$ be the smallest closed cone in $H^{2k}(X, \BR)$ containing the images of effective cycles of codimension $k$.
All cones in this paper are convex.

When we say a divisor/$\BQ$-divisor/$\BR$-divisor $D$ belongs to one of the above cones, we mean its image in $H^2(X,\BR)$ lies on the cone. For example, we call $D$ {\em pseudo effective} if $D\in \NE^1(X)$, which means that the image of $D$ in $H^2(X,\BR)$ lies on the cone $\NE^1(X)$.

Next, let us recall the concepts of Zariski decomposition \cite{ZariskiLinearSeries} and Mori dream space \cite{H-K}.

\begin{defn}\label{DEFZD}
Every pseudo effective $\BQ$-divisor $D$ on a smooth projective surface $X$ can be uniquely written as
\begin{equation}\label{GSRING2E021}
D = P + N
\end{equation}
in $\Pic_\BQ(X)$, where $P$ and $N$ are $\BQ$-divisors, $P$ is nef, $N$ is effective, $\supp(N)$ has negatively definite intersection matrix and $PN = 0$.
This is called the {\em Zariski decomposition} of $D$.
\end{defn}

\begin{defn}\label{DEFMDS}
A normal $\BQ$-factorial projective variety $X$ is a {\em Mori dream space} (MDS) if
\begin{enumerate}
\item[MD1.]
Every nef divisor on $X$ is semi-ample and
the nef cone $\NM^1(X)$ is generated by finitely many semi-ample divisors.
\item[MD2.]
There exists a finite collection of birational maps $f_i: X_i\dashrightarrow X$ for $i=1,2,...,a$
such that $f_i$ is an isomorphism in codimension one, $X_i$ is $\BQ$-factorial,
$\NM^1(X_i)$ is generated by finitely many semi-ample divisors and the moving cone of $X$ is given by
\begin{equation}\label{E523}
\Mov(X) = \bigcup_{i=1}^a (f_i)_* \NM^1(X_i).
\end{equation}
\end{enumerate}
\end{defn}

We need a version of Zariski decomposition in MDS.
There are various generalizations of Zariski decomposition to higher dimensions. For an effective $\BQ$-divisor $D$, one naive approach is to write $nD$ as the sum of its moving and fixed parts for $n\in \BZ^+$ such that $nD\in \Pic(X)$ and then take the limit as $n\to\infty$.
More precisely, we let
\begin{equation}\label{E509}
\begin{aligned}
F &= \sum_{G} \Big(\inf_{Q\in \Pi_D} \nu_G(Q) \Big) G
\\
\text{for }\Pi_D &= \Big\{Q \in Z_\BQ^1(X) \text{ effective}: 
D - Q\text{ is $\BQ$-movable}\Big\}
\end{aligned}
\end{equation}
where $G$ runs over all integral (i.e. reduced and irreducible) divisors on $X$ and $\nu_G(Q)$ is the multiplicity of $G$ 
in $Q$.

\begin{lem}\label{GSRING2LEM000}
For every effective $\BQ$-divisor $D$ on a normal $\BQ$-factorial projective variety $X$ with $F$ defined by \eqref{E509}, $F$ and $\Delta = D-F$ are effective $\BR$-divisors and $\Delta\in \Mov(X)$. If $\Mov(X)$ is rational polyhedral
and every $\BQ$-divisor in $\Mov(X)$ is $\BQ$-movable, then $F$ is an effective $\BQ$-divisor.
\end{lem}

\begin{proof}
Clearly, $\supp(F)\subset \supp(D)$. So $\nu_G(F) = 0$ for all but finitely many $G$ and $F$ is an effective $\BR$-divisor. And since $\nu_G(F) \le \nu_G(D)$ for all $G$, $\Delta$ is also effective.

Without loss of generality, let us assume that $D$ is an effective divisor. Let us write
\begin{equation*}\label{GSRING2E042}
n D = \Delta_n + F_n
\end{equation*}
in $\Pic(X)$ for all $n\in \BZ^+$, where
$\Delta_n$ and $F_n$ are the moving part and the fixed part of $|nD|$, respectively.
It is easy to see that
\begin{equation}\label{E530}
\frac{\nu_G(F_{ab})}{ab} \le \min\left(\frac{\nu_G(F_a)}{a}, \frac{\nu_G(F_b)}{b}\right)
\end{equation}
for all integral divisors $G$ and $a,b\in \BZ^+$ and
\begin{equation}\label{GSRING2E046}
F = \sum_{G} \Big(\inf_{n} \frac{\nu_G(F_n)}n \Big) G.
\end{equation}
So for every irreducible component $G\subset \supp(D)$, there exists an increasing sequence $\{l(G,i)\in \BZ^+: i=1,2,...\}$ such that
\begin{equation}\label{GSRING2E047}
\nu_G(F) = \lim_{i\to\infty}  \frac{\nu_G(F_{l(G,i)})}{l(G,i)}.
\end{equation}
Let
\begin{equation*}\label{GSRING2E048}
n_i = \prod_{\substack{G\subset \supp(D)\\ 1\le j\le i}} l(G,i).
\end{equation*}
By \eqref{E530}, \eqref{GSRING2E046} and \eqref{GSRING2E047}, we have
\begin{equation*}\label{GSRING2E049}
\frac{\nu_G(F_{n_1})}{n_1} \ge \frac{\nu_G(F_{n_2})}{n_2} \ge ...
\ge \nu_G(F) = \lim_{i\to\infty}  \frac{\nu_G(F_{n_i})}{n_i}
\end{equation*}
for all components $G\subset \supp(D)$. In other words,
\begin{equation*}\label{E510}
\frac{F_{n_1}}{n_1} \ge \frac{F_{n_2}}{n_2} \ge ...
\ge F = \lim_{i\to\infty} \frac{F_{n_i}}{n_i}
\end{equation*}
in $Z_\BR^1(X)$, where
we write $D_1 \le D_2$ ($D_2 \ge D_1$) if $D_2 - D_1$ is effective and
$D_1 < D_2$ ($D_2 > D_1$) if $D_1\le D_2$ and $D_1\ne D_2$.
It follows that 
\begin{equation*}\label{GSRING2E050}
\Delta = D - F = \lim_{i\to\infty} \frac{\Delta_{n_i}}{n_i} \in \Mov(X).
\end{equation*}

Suppose that $\Mov(X)$ is rational polyhedral and all $\BQ$-divisors in $\Mov(X)$ are $\BQ$-movable. 
Let $V$ be the closure of the set $\Pi_D$ defined by \eqref{E509}. More precisely, we let
\begin{equation*}\label{GSRING2E052}
V = \overline{\Big\{ Q = \sum_{i=1}^N x_i G_i: x_i\ge 0\in \BQ
\text{ and } D - Q\in \Mov(X)
\Big\}}
\end{equation*}
in $\BR^N = \{(x_1,x_2,...,x_N)\}$, where $G_1, G_2,..., G_N$ are the irreducible components of $\supp(D)$. It is easy to see that $V$ is a bounded closed convex set. And since $\Mov(X)$ is rational polyhedral,
$V$ must be a rational convex polyhedron.
By \eqref{E509}, $F\le Q$ for all $Q\in V$. So $F$ is an extremal point of $V$.
It follows that $F$ must be an effective $\BQ$-divisor.
\end{proof}

By the above lemma, every effective $\BQ$-divisor $D$ can be written as
\begin{equation*}\label{E524}
D = \Delta_D + F_D
\end{equation*}
with $F_D = F$ given by \eqref{E509}, where $\Delta_D\in \Mov(X)$ and $F_D$ are effective $\BR$-divisors. 
Alternatively, as in the proof of lemma, we may define $F_D$ as
\begin{equation*}\label{E527}
F_D = \inf_{nD\in \Pic(X)} \frac{F_{|nD|}}{n}
\end{equation*}
where $F_{|nD|}$ is the fixed part of the linear series $|nD|$.

When $X$ is a MDS, $\Mov(X)$ is rational polyhedral and every $\BQ$-divisor in $\Mov(X)$ is $\BQ$-movable
by MD2 in \ref{DEFMDS}.
So $\Delta_D$ and $F_D$ are effective $\BQ$-divisors and $\Delta_D$ is $\BQ$-movable.

\begin{defn}\label{GSRING2DEFPREC}
Let $V$ be a closed cone in $\BR^n$. For two points $p,q\in V$, we say that $p$ is more general than $q$ or $q$ is more special than $p$ in $V$, written as $p\succ_{_V} q$ or $q \prec_{_V} p$, if there exists $r > 0$ such that $rp - q\in V$. We will drop the notation $V$ and write $p\succ q$ or $q\prec p$ if $V$ is clear from the context.
\end{defn}

\begin{rem}\label{GSRING2REMPREC}
Let $V$ be a closed polyhedral cone in $\BR^n$. Let us recall that the dimension $r=\dim V$ of $V$ is the dimension of the smallest linear subspace $\Lambda$ of $\BR^n$ containing $V$ and the interior $V^\circ$ and the boundary $\partial V$ of $V$ are the interior and the boundary of $V$ in $\Lambda$, respectively. The faces of $V$ are subcones of $V$ defined as follows: Let $V = \cap_{j\in J} V_j$ be the intersection of finitely many half spaces $V_j$ with the corresponding hyperplane $\partial V_j$ containing the origin for $j\in J$. Then a face $W$ of $V$ is
the cone $\cap_{j\in I} \partial V_j\cap V$ for a subset $I\subset J$; when $I=\emptyset$, $W=V$ is regarded as the $r$-dimensional face of $V$ itself.

For every $p\in V$, there exists a unique face $W_p$ of $V$ such that $p\in W_p$ and $p\not\in W$ for any face $W$ of $V$ with $W\subsetneq W_p$.
It is easy to see that $p\succ q$ if and only if $W_p\supset W_q$.

For a point $q\in V$, $q\not\in W$ for any face $W$ of $V$ with $W\not\supset W_q$; otherwise, $q\in W\cap W_q\subsetneq W_q$. Therefore, there exists a minimal distance $\varepsilon > 0$ between $q$ and the faces $W$ of $V$ with $W\not\supset W_q$ 
since $V$ has only finitely many faces.
Then for all $p\in V$ and $||p - q|| < \varepsilon$,
we obviously have $W_p \supset W_q$ and hence $p\succ q$. In other words, the set $\{p\in V: p\succ q\}$ is open in $V$ for every point $q\in V$. On the other hand, the set $\{p\in V: p\prec q\}$ is the union of all faces $W$ of $V$ with $W\subset W_q$ and hence closed.
\end{rem}

\begin{lem}\label{GSRING2LEM001}
Let $X$ be a normal $\BQ$-factorial projective variety such that $\Mov(X)$ is rational polyhedral and every $\BQ$-divisor in $\Mov(X)$ is $\BQ$-movable.
Let $D$ be an effective $\BQ$-divisor on $X$.
\begin{enumerate}
\item
For all $\BR$-divisors $\widehat{\Delta}\in \Mov(X)$ and $\widehat{F} > 0$
satisfying that $\Delta_D \succ \widehat{\Delta}$ in $\Mov(X)$
and $\supp(F_D)\supset \supp(\widehat{F})$,
\begin{equation}\label{E526}
\widehat{D} = \widehat{\Delta} + \widehat{F}\not\in \Mov(X).
\end{equation}
\item 
For all $\BQ$-divisors $\widehat{\Delta}\in \Mov(X)$ and $\widehat{F} \ge 0$
satisfying that $\Delta_D \succ \widehat{\Delta}$ in $\Mov(X)$
and $\supp(F_D) \supset \supp(\widehat{F}) $,
\begin{equation}\label{E525}
\Delta_{\widehat{D}} = \widehat{\Delta}
\text{ and }
F_{\widehat{D}} = \widehat{F}
\end{equation}
for $\widehat{D} = \widehat{\Delta} + \widehat{F}$.
\item
For every effective $\BQ$-divisor $E$,
\begin{equation}\label{E531}
\lim_{n\to \infty}
\frac{\Delta_{nD+E}}n = \Delta_D \text{ and }
\lim_{n\to \infty}
\frac{F_{nD+E}}n = F_D.
\end{equation}
In addition, there exists $N\in \BN$ such that
\begin{equation}\label{GSRING2E057}
\supp(F_{nD + E}) \supset \supp(F_D),
\end{equation}
\begin{equation}\label{GSRING2E051}
\Delta_{nD + E} \succ \Delta_D
\end{equation}
and
\begin{equation}\label{GSRING2E055}
F_{nD+E} \ge (n-N) F_D
\end{equation}
for all $n\ge N$.
\end{enumerate} 
\end{lem}

\begin{proof}[Proof of Lemma \ref{GSRING2LEM001} (1)]
Otherwise, suppose that $\widehat{D}\in \Mov(X)$.

First let us prove \eqref{E526} for $\BQ$-divisors $\widehat{\Delta}$ and $\widehat{F}$.
Since $\Delta_D \succ \widehat{\Delta}$ in $\Mov(X)$
and $\supp(\widehat{F}) \subset \supp(F_D)$,
there exists $\varepsilon \in \BQ^+$ such that 
$\Delta_D - \varepsilon \widehat{\Delta} \in \Mov(X)$ and
$\varepsilon \widehat{F} \le F_D$. Then
\begin{equation*}\label{E529}
D - (F_D - \varepsilon \widehat{F}) =
(\Delta_D - \varepsilon \widehat{\Delta})  + \varepsilon \widehat{D} \in \Mov(X)
\end{equation*}
with $0 \le F_D - \varepsilon \widehat{F} < F_D$, which contradicts the definition of $F_D$.

Next, let us deal with the general case that $\widehat{\Delta}$ and $\widehat{F}$ are $\BR$-divisors.
Let
\begin{equation*}\label{E532}
\begin{aligned}
W &= \Big\{ (\Delta, F):
\Delta_D \succ \Delta \in \Mov(X),\ F = \sum_{i=1}^N x_i G_i \ge 0
\\
&\hspace{60pt}\text{ and } \Delta + F \in \Mov(X)
\Big\}
\end{aligned}
\end{equation*}
in $H^2(X,\BR) \times \BR^N$, where $G_1,G_2,...,G_N$ are the irreducible components of $\supp(F_D)$. 
Since $\Mov(X)$ is a rational polyhedral cone, so is $W$. And since $(\widehat{\Delta}, \widehat{F})\in W$ for $\widehat{F}\ne 0$,
$W$ contains a point $(\Delta', F')\in H^2(X,\BQ) \times \BQ^N$
with $F'\ne 0$. We have proved such $D' = \Delta' + F'\not\in \Mov(X)$. Contradiction.
\end{proof}

\begin{proof}[Proof of Lemma \ref{GSRING2LEM001} (2)]
By the definition of $F_{\widehat{D}}$, we have 
$F_{\widehat{D}} \le \widehat{F}$. 
Since
\begin{equation*}\label{GSRING2E053}
\widehat{\Delta} + (\widehat{F} - F_{\widehat{D}}) = \widehat{D} - F_{\widehat{D}} = \Delta_{\widehat{D}}
\in \Mov(X)	
\end{equation*}
we must have $F_{\widehat{D}} = \widehat{F}$ by \eqref{E526} and hence \eqref{E525} follows.
\end{proof}

\begin{proof}[Proof of Lemma \ref{GSRING2LEM001} (3)]
Clearly,
\begin{equation*}\label{GSRING2E054}
\varlimsup_{n\to \infty}\frac{\nu_G(F_{nD+E})}n \le \varlimsup_{n\to \infty}\frac{\nu_G(F_{nD}) + \nu_G(E)}n = \nu_G(F_D)
\end{equation*}
for all integral divisors $G$. Let $\{n_i: i\in \BZ^+\}$ be an increasing sequence of positive integers such that the limit
\begin{equation*}\label{E534}
F' = \lim_{i\to \infty} \frac{F_{n_i D+E}}{n_i}  \le F_D
\end{equation*}
exists. To prove \eqref{E531}, it suffices to show that $F' = F_D$. Clearly, 
\begin{equation*}\label{E533}
D - F' = \lim_{i\to \infty}\frac{n_i D + E - F_{n_i D+E}}{n_i}  \in \Mov(X)
\end{equation*}
and hence
\begin{equation*}\label{E535}
D - F' = \Delta_D + (F_D - F') \in \Mov(X).
\end{equation*}
It follows from \eqref{E526} that $F' = F_D$. This proves \eqref{E531}.

Clearly, \eqref{GSRING2E057} follows directly from \eqref{E531}.
Combining \eqref{E531} with the argument in Remark \ref{GSRING2REMPREC}, we obtain \eqref{GSRING2E051}.

Let us choose $N$ such that \eqref{GSRING2E057}
and \eqref{GSRING2E051} hold for $n\ge N$.
Then applying \eqref{E525}, we have
\begin{equation*}\label{GSRING2E060}
F_{n D + E} = F_{ND + E} + (n-N) F_D
\end{equation*}
for all $n\ge N$. 
This proves \eqref{GSRING2E055}.
\end{proof}

Now we are ready to prove Theorem \ref{THM001}.

Let us assume that $D_1, D_2,...,D_l$ are effective $\BQ$-divisors.
Clearly,
\begin{equation*}\label{GSRING2E059}
\supp(F_{m_1 D_1 + m_2 D_2 + ... + m_l D_l}) \subset \supp(D_1 + D_2 + ... + D_l)
\end{equation*}
for all $m_i\in \BN$.
Let $X$ be a MDS whose moving cone $\Mov(X)$ is given by \eqref{E523}. For every face $A$ of one of the cones $(f_i)_* \NM^1(X_i)$ and
every effective divisor $B\in Z^1(X)$ with
$B\le \supp(\sum D_i)$, we define
\begin{equation}\label{GSRING2E056}
\begin{aligned}
V_{A,B} &= \Big\{
(m_1,m_2,...,m_l): m_i\ge 0\in \BQ,\ \Delta_{\sum m_i D_i}\in A^\circ\\
&\hspace{108pt} \text{ and }
\supp(F_{\sum m_i D_i}) = B
\Big\}
\end{aligned}
\end{equation}
and let $\overline{V}_{A,B}$ be the closure of $V_{A,B}$ in $\BR^l$.
If $V_{A,B}\ne \emptyset$, then by \eqref{E525}, $\overline{V}_{A,B}$ is the preimage of the rational polyhedral cone
\begin{equation*}\label{GSRING2E058}
\Big\{
\Delta + F: \Delta \in A,\ F\ge 0\text{ and } \supp(F)\subset \supp(B)
\Big\} \subset H^2(X, \BR)
\end{equation*}
under the map
\begin{equation*}\label{GSRING2E062}
\varphi(m_1,m_2,...,m_l) = \sum_{i=1}^l m_i D_i. 
\end{equation*}
Therefore, $\overline{V}_{A,B}$ is a rational polyhedral cone if $V_{A,B}\ne \emptyset$.

Obviously, there are only finitely many pairs $(A,B)$ and
\begin{equation*}\label{GSRING2E061}
\BN^l = \bigcup_{A,B} (\overline{V}_{A,B} \cap \BN^l)
\end{equation*}
with each $\overline{V}_{A,B} \cap \BN^l$ being a submonoid of $\BN^l$. For each rational polyhedral cone $V\subset \BR^l$, we let
\begin{equation}\label{GSRING2E063}
\begin{aligned}
R(X,V) &= \bigoplus_{\mathbf{v}\in V\cap \BN^l}
H^0(\varphi(\mathbf{v}))
\text{ and}\\
M(X,D,V) &= \bigoplus_{\mathbf{v}\in V\cap \BN^l}
H^0(D+\varphi(\mathbf{v}))
\end{aligned}
\end{equation}
be the subring and the submodule of $R(X,D_1,...,D_l)$ and
$M(X,D,D_1,...,D_l)$, respectively. Under these notations, we have
\begin{equation*}\label{GSRING2E064}
\begin{aligned}
R(X,D_1,D_2,...,D_l) &= \sum_{A,B} R(X,\overline{V}_{A,B}) 
\text{ and}\\
M(X,D,D_1,D_2...,D_l) &= \sum_{A,B} M(X,D,\overline{V}_{A,B}).
\end{aligned}
\end{equation*}
Thus, to prove the finite generation of $R(X,D_1,...,D_l)$ and
$M(X,D,D_1,...,D_l)$, it suffices to show the same for
$R(X,V)$ and $M(X,D,V)$ with $V = \overline{V}_{A,B}$. That is, it comes down to proving
\begin{equation}\label{GSRING2E065}
\begin{aligned}
& R(X,V)  
\text{ is a finitely generated $\BC$-algebra and}\\
& M(X,D,V) \text{ is a finitely generated module over } R(X,V)
\end{aligned}
\end{equation}
for each $V = \overline{V}_{A,B}$.

Each cone $V = \overline{V}_{A,B}$ has the following properties:
\begin{itemize}
\item
Since $A$ is a face of $(f_i)_* \NM^1(X_i)$ for some $i$, we may assume that $A\subset \NM^1(X)$ after replacing $X$ by $X_i$. And since every nef divisor on $X$ is semi-ample, we have
\begin{equation}\label{GSRING2E066}
\Delta_{\varphi(\mathbf{v})}\text{ is semi-ample for all } \mathbf{v} \in V_\BQ.
\end{equation}
\item
By \eqref{GSRING2E056}, we have
\begin{equation}\label{GSRING2E067}
F_{\varphi(\mathbf{v}_1 + \mathbf{v}_2)} = 
F_{\varphi(\mathbf{v}_1)} + F_{\varphi(\mathbf{v}_2)}
\text{ for all } \mathbf{v}_1, \mathbf{v}_2 \in V_\BQ.
\end{equation}
\end{itemize}

We can further divide each $\overline{V}_{A,B}$ into a finite union of rational simplicial cones. So it suffices to prove \eqref{GSRING2E065} for every rational simplicial cone $V\subset \BR^l$ satisfying
\eqref{GSRING2E066} and \eqref{GSRING2E067}.

Let us choose a set of generators $\mathbf{v}_1, \mathbf{v}_2,...,\mathbf{v}_r$ for the rational simplicial cone $V$ such that
\begin{itemize}
\item
$V = \{x_1 \mathbf{v}_1 + x_2 \mathbf{v}_2 + ... + x_r \mathbf{v}_r: x_i\ge 0\}$, $\dim V = r-1$,
\item
$\mathbf{v}_i\in V\cap \BZ^l$, $L_i = \varphi(\mathbf{v}_i)\in \Pic(X)$,
\item
$\Delta_{L_i}$ is base point free, $F_{L_i}$ is effective and
\begin{equation}\label{GSRING2E068}
F_{L_1 + L_2 + ... + L_r} = 
F_{L_1} + F_{L_2} + ... + F_{L_r}
\end{equation}
\end{itemize}
for $i=1,2,...,r$. Note that \eqref{GSRING2E067} and \eqref{GSRING2E068} are actually equivalent by \eqref{E525}.

Every $\mathbf v\in V\cap \BN^l$ can be written in a unique way as
\begin{equation*}\label{GSRING2E069}
\mathbf v = \mathbf u + n_1 \mathbf{v}_1 + n_2 \mathbf{v}_2 + ... + n_r \mathbf{v}_r
\end{equation*}
for some $n_i\in \BN$ and some
\begin{equation*}\label{GSRING2E070}
\mathbf u\in \Lambda = \{ x_1 \mathbf{v}_1 + x_2 \mathbf{v}_2 + ... + x_r \mathbf{v}_r: 0\le x_i < 1\}
\cap \BN^l.
\end{equation*}
Then we can rewrite \eqref{GSRING2E063} as
\begin{equation*}\label{GSRING2E071}
\begin{aligned}
R(X,V) &= \bigoplus_{\substack{\mathbf{u}\in \Lambda\\ n_1,n_2,...,n_r\in \BN}}
H^0(\varphi(\mathbf{u}) + \sum_{i=1}^r n_i L_i)
\text{ and}\\
M(X,D,V) &= \bigoplus_{\substack{\mathbf{u}\in \Lambda\\ n_1,n_2,...,n_r\in \BN}}
H^0(D + \varphi(\mathbf{u}) + \sum_{i=1}^r n_i L_i).
\end{aligned}
\end{equation*}
The finite generation of $R(X,V)$ and $M(X,D,V)$ will follow if
there exists a number $N$ such that the maps
\begin{equation*}\label{E528}
\begin{tikzcd}[column sep=small]
\displaystyle{H^0(\varphi(\mathbf{u}) + \sum_{i=1}^r n_i L_i) \otimes H^0(L_j)} \ar[two heads]{r} &
\displaystyle{H^0(\varphi(\mathbf{u}) + \sum_{i=1}^r n_i L_i + L_j)}\\
\displaystyle{H^0(D + \varphi(\mathbf{u}) + \sum_{i=1}^r n_i L_i) \otimes H^0(L_j)} \ar[two heads]{r} &
\displaystyle{H^0(D + \varphi(\mathbf{u}) + \sum_{i=1}^r n_i L_i + L_j)}
\end{tikzcd}
\end{equation*}
are surjective for $1\le j\le r$, $\mathbf{u}\in \Lambda$, $n_1,n_2,...,n_r\in \BN$ and
$n_j \ge N$. Note that $\Lambda$ is a finite set.

It suffices to prove the following:

\begin{prop}\label{GSRING2PROP000}
Let $X$ be a normal $\BQ$-factorial projective variety such that $\Mov(X)$ is rational polyhedral and every $\BQ$-divisor in $\Mov(X)$ is $\BQ$-movable and let $L_1, L_2,...,L_r$ be effective cartier divisors on $X$ such that
\begin{equation*}\label{E536}
\begin{aligned}
& \Delta_{L_i} \text{ is base point free, } H^0(\CO_X(F_{L_i})) \ne 0 \text{ for } i=1,2,...,r \text{ and}
\\
& F_{L_1 + L_2 + ... + L_r} = 
F_{L_1} + F_{L_2} + ... + F_{L_r}.
\end{aligned}
\end{equation*}
For every divisor $L\in \Pic(X)$, there exists $N\in \BN$ such that the map
\begin{equation}\label{GSRING2E072}
\begin{tikzcd}
\displaystyle{H^0(L + \sum_{i=1}^r n_i L_i) \otimes H^0(L_j)} \ar[two heads]{r} &
\displaystyle{H^0(L + \sum_{i=1}^r n_i L_i + L_j)}
\end{tikzcd}
\end{equation}
is surjective for $1\le j\le r$, $n_1,n_2,...,n_r\in \BN$ and $n_j \ge N$.
\end{prop}

\begin{proof}
We argue by induction on $r$.
For every $a\in \BZ^+$, by applying the induction hypothesis to $(L_1,...,\widehat{L}_m, ..., L_r, L+kL_m)$ for
each $m=1,2,...,r$ and $k=0,1,...,a-1$, we can find $N_a>a$ such that the map \eqref{GSRING2E072} is surjective for $1\le j\le r$, $n_1,n_2,...,n_r\in \BN$, $n_j \ge N_a$ and $\min\{n_i\} < a$. So it suffices to prove \eqref{GSRING2E072} for some $a$ and all $n_1,n_2,...,n_r\ge a$.

If $L+\sum n_i L_i$ is not effective for all $n_i\in \BN$, there is nothing to prove. Otherwise,
$L + b \sum L_i$ is effective for $b$ sufficiently large. And by \eqref{GSRING2E057} and \eqref{GSRING2E051}, we can find $b\in \BN$ such that
\begin{equation*}\label{GSRING2E073}
\begin{aligned}
\supp(F_{L + b(L_1+L_2+...+L_r)}) &\supset \supp(F_{L_1+L_2+...+L_r})\\
\Delta_{L + b(L_1+L_2+...+L_r)} &\succ \Delta_{L_1+L_2+...+L_r}.
\end{aligned}
\end{equation*}
By our hypothesis \eqref{GSRING2E068}, we have
\begin{equation*}\label{GSRING2E074}
\begin{aligned}
\supp(F_{L_1+L_2+...+L_r}) &= \supp(F_{L_1}) + \supp(F_{L_2}) + ... + \supp(F_{L_r})\\
\Delta_{L_1+L_2+...+L_r} &= \Delta_{L_1} + \Delta_{L_2} + ... + \Delta_{L_r}.
\end{aligned}
\end{equation*}
Consequently,
\begin{equation*}\label{GSRING2E075}
F_{L + n_1 L_1+ n_2 L_2+...+ n_r L_r} \ge \sum_{i=1}^r (n_i - b) F_{L_i}
\end{equation*}
for all $n_1,n_2,...,n_r\ge b$. It follows that
\begin{equation*}\label{GSRING2E076}
H^0(L + \sum_{i=1}^r n_i L_i) =  H^0( (L+b\sum_{i=1}^r L_i) + \sum_{i=1}^r (n_i - b) \Delta_{L_i})
\end{equation*}
for all $n_1,n_2,...,n_r\ge b$. So \eqref{GSRING2E072} becomes
\begin{equation}\label{GSRING2E077}
\begin{tikzcd}[column sep=small]
\displaystyle{H^0(G + \sum_{i=1}^r (n_i-b) \Delta_{L_i}) \otimes H^0(\Delta_{L_j})} \ar[two heads]{r} &
\displaystyle{H^0(G + \sum_{i=1}^r (n_i-b) \Delta_{L_i} + \Delta_{L_j})}
\end{tikzcd}
\end{equation}
for $G = L + b(L_1+L_2+...+L_r)$ with $\Delta_{L_i}$ base point free by our hypothesis.

Let $\pi$ be the morphism
\begin{equation*}\label{GSRING2E078}
\begin{tikzcd}
X \ar{r}{\pi} & 
\displaystyle{P = \prod_{i=1}^r \PP H^0(\Delta_{L_i})^\vee} 
\end{tikzcd}
\end{equation*}
given by the linear series $|\Delta_{L_1}| \times |\Delta_{L_2}| \times ... \times |\Delta_{L_r}|$. Then
\begin{equation*}\label{GSRING2E079}
H^0(G + \sum_{i=1}^r m_i \Delta_{L_i}) = H^0(\pi_* \CO_X(G) \otimes \CO_P(m_1,m_2,...,m_r))
\end{equation*}
for $m_i\in \BN$, where $\CO_P(m_1,m_2,...,m_r)$ is the line bundle on $P$ such that
$\pi^* \CO_P(m_1,m_2,...,m_r) = \CO_X(m_1 \Delta_{L_1} + m_2 \Delta_{L_2} + ... + m_r \Delta_{L_r})$.

By Serre vanishing, we can find $c,d,e\in \BN$ such that the map
\begin{equation*}\label{GSRING2E080}
\begin{tikzcd}
\CO_P(c,c,...,c)^{\oplus e} \ar[two heads]{r} & \f(d,d,...,d)
\end{tikzcd}
\end{equation*}
and the induced maps
\begin{equation*}\label{GSRING2E083}
\begin{tikzcd}
H^0(\CO_P(c+m_1,...,c+m_r))^{\oplus e} \ar[two heads]{r} & H^0(\f(d+m_1,...,d+m_r))
\end{tikzcd}
\end{equation*}
are surjections for all $m_i\in \BN$, where $\f = \pi_* \CO_X(G)$. Then we see that
\begin{equation}\label{GSRING2E081}
\begin{tikzcd}[column sep=small, row sep=6pt]
& H^0(\pi_* \CO_X(G)\otimes \CO_P(d_1,d_2,...,d_r)) \otimes H^0(\CO_P(m_1,m_2,...,m_r))\\
\ar[two heads]{r} & 
H^0(\pi_* \CO_X(G)\otimes \CO_P(d_1+m_1,d_2+m_2,...,d_r+m_r))
\end{tikzcd}
\end{equation}
is surjective for all $d_i\ge d, m_i\in \BN$ by the diagram
\begin{equation*}\label{GSRING2E082}
\begin{tikzcd}
H^0(\CO_P(c_i))^{\oplus e}\otimes H^0(\CO_P(m_i)) \ar{d}
\ar[two heads]{r} & H^0(\CO_P(c_i+m_i))^{\oplus e}\ar[two heads]{d}\\
H^0(\f(d_i)) \otimes H^0(\CO_P(m_i)) 
\ar{r} & 
H^0(\f(d_i+m_i))
\end{tikzcd}
\end{equation*}
where we write
\begin{equation*}\label{GSRING2E084}
\begin{aligned}
c_i &= d_i + (c-d),\\
\CO_P(c_i) &= \CO_P(c_1,c_2,...,c_r),\\
\CO_P(m_i) &= \CO_P(m_1,m_2,...,m_r),\\
\CO_P(c_i + m_i) &= \CO_P(c_1+m_1,c_2+m_2,...,c_r+m_r),\\
\f(d_i) &= \f(d_1,d_2,...,d_r) \text{ and}\\
\f(d_i+m_i) &= \f(d_1+m_1,d_2+m_2,...,d_r+m_r).
\end{aligned}
\end{equation*}
Then \eqref{GSRING2E077} follows from \eqref{GSRING2E081} for $n_i - b \ge d$. It suffices to take $a = b+d$.
\end{proof}

\section{Zariski's Conjecture and Rationality of Poincar\'e Series}\label{GSRING2SEC003}

\subsection{Sketch of the proof of Theorem \ref{GSRING2THM000}}

We are going to prove Theorem \ref{GSRING2THM000} in this section.
Here is an outline of our proof:
\begin{itemize}
\item
First, we prove it for all $D_i$ big and nef, assuming Conjecture \ref{GSRING2CONJ000}. Here we follow closely the argument of Cutkosky and Srinivas \cite{ZariskiConjecture} (see also \cite{GorlachZariski}).
\item
Next, we prove it for all $D_i = P_i + N_i$ satisfying
\begin{equation}\label{GSRING2E025}
\left(\sum P_i \right) \left(\sum N_i \right)  = 0
\end{equation}
where $D_i = P_i + N_i$ is the Zariski decomposition of $D_i$. We say that $D_1, D_2, ..., D_l$ have compatible Zariski decomposition if \eqref{GSRING2E025} holds.
\item
Finally, we reduce the case $D_i$ big to \eqref{GSRING2E025} using 
the theory of augmented and restricted base loci \cite{ELMN} and
some elementary convex geometry.
\end{itemize}

\subsection{Reduction to Conjecture \ref{GSRING2CONJ000}}

Suppose that all $D_i$ are big and nef. Let us argue by induction on $l$.

For each $N\in \BN$, we can write
\begin{equation*}\label{E522}
\begin{aligned}
M_{X,D,D_1,D_2,...,D_l}(t)
&= \sum_{m_1,m_2,...,m_l\ge N} h^0(D + \sum m_i D_i) t^m\\
&+\sum_{\substack{I\subset \{1,2,...,l\}\\
I\ne \emptyset}} (-1)^{|I|+1} \sum_{\substack{m_i < N\\ i\in I}}
h^0(D + \sum m_i D_i) t^m
\end{aligned}
\end{equation*}
where $t = (t_1,t_2,...,t_l)$ and $t^m
= t_1^{m_1} t_2^{m_2} ... t_l^{m_l}$.

For each $I\subset\{1,2,...,l\}$ and $I\ne \emptyset$, we have
\begin{equation*}\label{GSRING2E008}
\begin{aligned}
&\quad \sum_{\substack{m_i < N\\ i\in I}} h^0(D + \sum m_i D_i) t^m
\\
&= \sum_{\substack{m_i < N\\ i\in I}} \prod_{i\in I} t_i^{m_i} \Big(\sum_{\substack{m_j\in \BN\\ j\not\in I}}
h^0((D + \sum_{i\in I} m_i D_i) + \sum_{j\not\in I} m_j D_j)
\prod_{j\not\in I} t_j^{m_j}\Big)
\end{aligned}
\end{equation*}
By induction on $l$, we may assume that
\begin{equation*}\label{GSRING2E009}
\sum_{\substack{m_j\in \BN\\ j\not\in I}} h^0((D + \sum_{i\in I} m_i D_i) + \sum_{j\not\in I} m_j D_j)
\prod_{j\not\in I} t_j^{m_j}
\end{equation*}
is rational. Therefore,
\begin{equation*}\label{GSRING2E010}
\sum_{\substack{m_i < N\\ i\in I}} h^0(D + \sum m_i D_i) t^m
\end{equation*}
is rational for all $N\in \BN$ and $I\ne\emptyset \subset\{1,2,...,l\}$. So, to prove the rationality 
of $M_{X,D,D_1,D_2,...,D_l}(t)$,
it suffices to show the rationality of
\begin{equation}\label{GSRING2E011}
\sum_{m_1,m_2,...,m_l\ge N} h^0(D + \sum m_i D_i) t_1^{m_1}t_2^{m_2}
... t_l^{m_l}
\end{equation}
for some $N$. In other words, we can ``chop off'' the part of the series $M_{X,D,D_1,D_2,...,D_l}(t)$ of degree $< N$ in one of $t_i$. Of course, we want to choose $N$ sufficiently large.

Since $D_i$ are nef, there exists an ample divisor $A$ on $X$ such that
\begin{equation}\label{GSRING2E012}
H^n(X, A + D + \sum m_i D_i) = 0
\end{equation}
for all $n\ge 1$ and $m_i\ge 0$. If one applies Kawamata-Viehweg vanishing (cf. \cite{E-V}), it suffices to choose $A$ such that
$-K_X + A + D$ is ample. Alternatively, \eqref{GSRING2E012}
follows from Fujita's vanishing on surfaces, which works in any characteristic \cite{FujitaSemipositive}. 

Since $D_i$ are big, there exists $m\in \BZ^+$ such that
\begin{equation*}\label{GSRING2E013}
H^0(X, m \sum D_i - A) \ne 0.
\end{equation*}
Let $C\in |m \sum D_i - A|$. Then we have a short exact sequence
\begin{equation}\label{GSRING2E014}
\begin{tikzcd}[cells={anchor=center}]
& \CO_X(A+D + \sum (m_i-m) D_i)\ar[equal]{d}\\
0 \ar{r} & \CO_X(D + \sum m_i D_i - C) \ar{r} & \CO_X(D + \sum m_i D_i)
\\
\phantom{0}\ar{r} & \CO_C(D + \sum m_i D_i) \ar{r} & 0.
\end{tikzcd}
\end{equation}
For $m_i\ge m$, \eqref{GSRING2E012} and \eqref{GSRING2E014}
induce isomorphisms
\begin{equation}\label{GSRING2E015}
\begin{tikzcd}
H^n(\CO_X(D + \sum m_i D_i))\ar{r}{\simeq} & H^n(\CO_C(D + \sum m_i D_i))
\end{tikzcd}
\end{equation}
when $n\ge 1$.

We write $C = C_0 + C_1$ such that 
$C_i$ are effective, $C_0 \sum D_i = 0$ and
$\Gamma \sum D_i > 0$ for all irreducible components $\Gamma$ of $C_1$. Then by the exact sequence
\begin{equation*}\label{GSRING2E016}
\begin{tikzcd}[cells={anchor=center}]
0 \ar{r} & \CO_{C_1}(D + \sum m_i D_i - C_0) \ar{r} & \CO_C(D + \sum m_i D_i)
\\
\phantom{0}\ar{r} & \CO_{C_0}(D + \sum m_i D_i) \ar{r} & 0,
\end{tikzcd}
\end{equation*}
we see that
\begin{equation*}\label{GSRING2E017}
\begin{tikzcd}[cells={anchor=east}]
H^n(\CO_X(D + \sum m_i D_i))\ar{r}{\simeq} & H^n(\CO_C(D + \sum m_i D_i))\\
\phantom{0} \ar{r}{\simeq} & H^n(\CO_{C_0}(D + \sum m_i D_i))
\end{tikzcd}
\end{equation*}
for $n\ge 1$ and $\min(m_1,m_2,...,m_l) >> 1$.
Replacing $C$ by $C_0$, we have $C \sum D_i = 0$ and \eqref{GSRING2E015}.

In conclusion, there exist $N\in \BZ^+$ and an effective divisor $C$
such that
\begin{equation*}\label{GSRING2E018}
\begin{aligned}
\sum CD_i &= 0 \text{ and}\\
h^n(\CO_X(D + \sum m_i D_i)) &= h^n(\CO_{C}(D + \sum m_i D_i))
\end{aligned}
\end{equation*}
for $n\ge 1$ and $m_1,m_2,...,m_l\ge N$.

Obviously, the rationality of \eqref{GSRING2E011} follows from the rationality of
\begin{equation}\label{GSRING2E019}
\begin{aligned}
&\quad\sum_{m_1,m_2,...,m_l\ge N} h^n(\CO_X(D + \sum m_i D_i)) t_1^{m_1}t_2^{m_2}
... t_l^{m_l}
\\
&= \sum_{m_1,m_2,...,m_l\ge N} h^n(\CO_C(D + \sum m_i D_i)) t_1^{m_1}t_2^{m_2}
... t_l^{m_l}
\end{aligned}
\end{equation}
for all $n\ge 1$, combined with Riemann-Roch. 
Furthermore, since $C$ is supported on the union $\CC$ of integral curves
$\Gamma$ with $\Gamma D_i = 0$ and
there exists a birational morphism $f: X' \to X$ such that $f^{-1}(\CC)$ has simple normal crossings (for short, snc), we may assume that $C$ has snc support after replacing $X$ by $X'$.
Therefore, the rationality of \eqref{GSRING2E019} follows from Conjecture
\ref{GSRING2CONJ000}.

At the moment, Conjecture \ref{GSRING2CONJ000} is out of reach for us. 
It was proved in \cite[Theorem 8]{ZariskiConjecture} for $a = 1$, as a consequence of a key result of Cutkosky and Srinivas \cite[Theorem 7]{ZariskiConjecture}:

\begin{thm***}[Cutkosky-Srinivas]\label{THMCSDYNAMICS}
Let $G$ be a connected commutative algebraic group defined
over an algebraically closed field $k$ of characteristic $0$. Suppose that $x \in G(k)$
is such that the cyclic subgroup $\langle x\rangle = \{n\cdot x | n \in \BZ \}$ is Zariski dense in $G$. Then
any infinite subset of $\langle x\rangle$ is Zariski dense in $G$.
\end{thm***}

This is a deep result and, in our opinion,
the most crucial step of Cutkosky and Srinivas' proof
of Zariski's conjecture. Unfortunately, we are unable to generalize
this to prove Conjecture \ref{GSRING2CONJ000}, even under the assumption of Mordell-Lang. A further discussion of this conjecture is planned in a future paper.

In summary, this proves Theorem \ref{GSRING2THM000} when $D_i$ are big and nef, assuming 
Conjecture \ref{GSRING2CONJ000}. For $D_i$ big but not necessarily nef, we are going to use Zariski's decomposition to reduce it to the nef case just as in \cite{ZariskiConjecture}.

\subsection{Zariski decomposition via stable loci}

A more contemporary interpretation of \eqref{GSRING2E021} \cite{ELMN} is
\begin{equation*}\label{GSRING2E022}
\begin{aligned}
\supp(N) &= \mathbf{B}_-(D) = \bigcup_{A} \mathbf{B}(D+A)\\
&= \bigcup_{A}
\Big(\bigcap_{m} \Bs(m(D + A))\Big)
\end{aligned}
\end{equation*}
where $A$ runs over all ample $\BQ$-divisors, $m$ runs over all positive integers such that $m(D+A)$ is integral,
$\Bs(L)$ is the base locus of the linear series $|L|$,
$\mathbf{B}(L)$ is called the {\em stable locus} of $L$ and $\mathbf{B}_-(L)$ is called the {\em restricted base locus} of $L$. 

When $D$ is big in \eqref{GSRING2E021}, $P$ is big and
\begin{equation}\label{GSRING2E026}
\begin{aligned}
\supp(N) &= \mathbf{B}_-(D) \subset \mathbf{B}_+(D) =
\bigcap_{A} \mathbf{B}(D-A)\\
&= \mathbf{B}_+(P) = \bigcup_{\substack{C\subset X \text{ integral curve}\\ PC=0}} C
\end{aligned}
\end{equation}
where $A$ runs over all ample $\BQ$-divisors and
$\mathbf{B}_+(L)$ is called the {\em augmented base locus} of $L$ \cite[Example 1.11]{ELMN}.

Both restricted and augmented base loci can be defined for $\BR$-divisors \cite{ELMN}. When we define $\mathbf{B}_-(D)$ or
$\mathbf{B}_+(D)$ for $\BR$-divisors $D$, we choose the corresponding $\BR$-ample divisor $A$ such that $D+A$ or $D-A$ is a $\BQ$-divisor. The same statements \eqref{GSRING2E022}
and \eqref{GSRING2E026} hold for big $\BR$-divisors on a smooth projective surface \cite[Example 3.4]{ELMN}.


\begin{lem}\label{GSRING2LEMZARISKI}
Let $D$ be a big $\BR$-divisor on a smooth projective surface $X$ with Zariski decomposition $D = P + N$. There exists $\varepsilon > 0$ such that for every $\BR$-divisor $D' = P' + N'$ on $X$ satisfying
$|| D - D' || < \varepsilon$, $PN' = P'N = 0$ and $\supp(N)\subset \supp(N')$,
where $D' = P' + N'$ is the Zariski decomposition of $D'$ and $||F||$ is the Euclidean norm on $H^2(X,\BR)$.
\end{lem}

\begin{proof}
By \cite[Corollary 1.6]{ELMN}, there exists $\varepsilon_1 > 0$ such that
\begin{equation*}\label{GSRING2E027}
\mathbf{B}_+(D') \subset \mathbf{B}_+(D) 
\end{equation*}
for all $D'$ satisfying $|| D - D' || < \varepsilon_1$. Therefore,
\begin{equation*}\label{GSRING2E029}
\supp(N') \subset \mathbf{B}_+(D') \subset \mathbf{B}_+(D)
= \mathbf{B}_+(P)
\end{equation*}
and hence $PN' = 0$.

Since $\mathbf{B}_-(D) = \supp(N)$ is a scheme (it is a countable union of irreducible subvarieties of $X$ in general), there exists an ample $\BR$-divisor $A$ such that $\mathbf{B}_-(D) = \mathbf{B}(D+A)$. Then there exists $\varepsilon_2 > 0$ such that
$A + F$ is ample for all $\BR$-divisors $F$ satisfying
$||F|| < \varepsilon_2$. It follows that
\begin{equation*}\label{GSRING2E028}
\begin{aligned}
\supp(N) &= \mathbf{B}_-(D) = \mathbf{B}(D+A)
\\
&= \mathbf{B}(D' + A + (D - D')) \subset \mathbf{B}_-(D') = \supp(N')\\
& \subset \mathbf{B}_+(D') = \mathbf{B}_+(P')
\end{aligned}
\end{equation*}
for all $D'$ satisfying $||D-D'||<\varepsilon_2$ and hence $P'N = 0$.

It suffices to take $\varepsilon = \min(\varepsilon_1, \varepsilon_2)$ to finish the proof of the lemma.
\end{proof}

Now let us prove the theorem under the hypothesis \eqref{GSRING2E025}.

Let $s$ be a positive integer such that $sP_i$ and $sN_i$ are integral for all $i=1,2,...,l$. We write $m_i = sq_i + r_i$ for
$q_i\in \BN$ and $0\le r_i < s$. Then
\begin{equation*}\label{GSRING2E023}
\begin{aligned}
M_{X,D,D_1,D_2,...,D_l}(t) &= \sum_{0\le r_i < s}
t_1^{r_1}t_2^{r_2}...t_l^{r_l}
M_{X,D+\sum r_iD_i, sD_1,sD_2,...,sD_l}(t^s)
\end{aligned}
\end{equation*}
where $t^s = (t_1^s,t_2^s,...,t_l^s)$. So this reduces it to the case that $P_i$ and $N_i$ are integral.

By Lemma \ref{GSRING2LEMZARISKI}, there exists $n\in \BZ^+$ such that
the Zariski decomposition
\begin{equation*}\label{GSRING2E030}
D + n\sum_{i\in l} D_i = P + N
\end{equation*}
satisfies that
\begin{equation*}\label{GSRING2E031}
\left(P + \sum P_i\right)\left(N + \sum N_i\right) = 0.
\end{equation*}
Therefore,
\begin{equation*}\label{GSRING2E032}
h^0(D + \sum m_i D_i) = h^0((D+ n\sum D_i) + \sum (m_i-n)P_i) 
\end{equation*}
for $m_1,m_2,...,m_l\ge n$. Since we have proved the theorem for $D_i$ nef,
\begin{equation*}\label{GSRING2E033}
\begin{aligned}
&\quad \sum_{m_1,m_2,...,m_l\ge n}
h^0(D + \sum m_i D_i) t_1^{m_1} t_2^{m_2} ... t_l^{m_l}
\\
&= \sum_{m_1,m_2,...,m_l\ge n}
h^0((D+ n\sum D_i) + \sum (m_i-n)P_i) t_1^{m_1} t_2^{m_2} ... t_l^{m_l}
\end{aligned}
\end{equation*}
is rational. This proves the theorem under the hypothesis \eqref{GSRING2E025}.

\subsection{Zariski decomposition chambers}

Finally, let us reduce the general case to 
\eqref{GSRING2E025}. Let $W\subset H^2(X,\BR)$ be the cone generated 
by $D_1,D_2,...,D_l$. Since all $D_i$ are big, every nonzero $\BR$-divisor $D\in W$ is big and hence $\mathbf{B}_-(D)$ is a union of curves contained in
\begin{equation}\label{GSRING2E034}
\mathbf{B}_-(D) \subset \mathbf{B}_+(D) \subset \bigcup_{i=1}^l \mathbf{B}_+(D_i) = G
\end{equation}
where $G$ is a reduced effective divisor on $X$. For every effective divisor $\Gamma\le G$, we let
$W_\Gamma$ be the set of $D\in W$ such that $\mathbf{B}_-(D) = \Gamma$. Then $W_\Gamma$ is a (not necessarily closed) cone. We claim that its closure $\overline{W}_\Gamma$ is rational polyhedral.

\begin{lem}\label{GSRING2LEMZDC}
Let $X$ be a smooth projective surface, $W\subset H^2(X, \BR)$ be a cone generated by finitely many big divisors on $X$, $\Gamma$ be a reduced effective divisor and $W_\Gamma$
be the set of $D\in W$ such that $\mathbf{B}_-(D) = \Gamma$. Then $\overline{W}_\Gamma$ is a rational polyhedral cone.
\end{lem}

\begin{proof}
There is nothing to prove if $W_\Gamma = \emptyset$. Suppose that $W_\Gamma\ne \emptyset$.
We claim that
\begin{equation}\label{GSRING2E211}
\overline{W}_\Gamma = \{
D\in W: D = P + N,\ P\Gamma = 0,\ \supp(N) \subset \Gamma
 \}
\end{equation}
where $D = P + N$ is the Zariski decomposition of $D$.

We choose some $D' = P' + N' \in W_\Gamma$.
For each $D = P+N\in W$ satisfying $P\Gamma = 0$ and $\supp(N)\subset \Gamma$, $D$ and $D'$ have compatible Zariski decompositions and hence
$s D + t D'$ has Zariski decomposition
\begin{equation*}\label{GSRING2E212}
sD + tD' = (sP + tP') + (sN + tN')\in W_\Gamma
\end{equation*}
for all $s, t > 0$. This shows that $D\in \overline{W}_\Gamma$.

On the other hand, for every $\BR$-divisor $D = P + N\in \overline{W}_\Gamma$, we can find $\varepsilon > 0$ as in Lemma \ref{GSRING2LEMZARISKI} such that $P N' = 0$ and $\supp(N)\subset \supp(N')$ for
all $D' = P' + N'$ satisfying $||D - D'|| < \varepsilon$. Clearly, we can find $D'\in W_\Gamma$ such
that $|| D- D' || < \varepsilon$ and it follows that $P\Gamma = 0$ and $\supp(N)\subset \Gamma$.
This proves \eqref{GSRING2E211}.

Let $\Lambda_\Gamma$ be the linear subspace of $H^2(X, \BR)$ spanned by the irreducible components of $\Gamma$ and let $\sigma_\Gamma: H^2(X, \BR)\to H^2(X, \BR) \times H^2(X,\BR)$ be the map given by
the projections of $H^2(X, \BR)$ to $\Lambda_\Gamma$ and $\Lambda_\Gamma^\perp$ under the cup product. That is,
\begin{equation*}\label{GSRING2E213}
\begin{aligned}
\sigma_\Gamma(D) &= (F_1, F_2)\text{ for } D = F_1 + F_2 \text{ with } F_1\in \Lambda_\Gamma
\text{ and}\\
&\hspace{144pt} F_2 E = 0 \text{ for all } E\in \Lambda_\Gamma.
\end{aligned}
\end{equation*}
Note that since we assume that $W_\Gamma \ne \emptyset$, the components of $\Gamma$ have negative definite self-intersection matrix. So the decomposition $D = F_1 + F_2$ is unique for each
$D\in H^2(X, \BR)$ and $\sigma_\Gamma$ is well defined.

Obviously, $\sigma_\Gamma$ is $\BQ$-linear. Namely, it is induced by the corresponding linear map
$H^2(X, \BQ)\to H^2(X, \BQ) \times H^2(X,\BQ)$.

We let $S_\Gamma$ be the cone in $H^2(X,\BR)$ generated by the irreducible components of $\Gamma$
and let $T_G$ be the cone in $H^2(X,\BR)$ given by
\begin{equation*}\label{GSRING2E214}
T_G = \{ D: D C \ge 0\text{ for all irreducible components } C\subset G
\}
\end{equation*}
where $G$ is given by \eqref{GSRING2E034} with $D_i$ the generators of $W$.

Clearly, both $S_\Gamma$ and $T_G$ are rational polyhedral cones. And since $\sigma_\Gamma$ is
$\BQ$-linear,
$\sigma_\Gamma^{-1}(S_\Gamma\times T_G)$ is also a rational polyhedral cone. We claim that
\begin{equation}\label{GSRING2E215}
\overline{W}_\Gamma = W\cap \sigma_\Gamma^{-1}(S_\Gamma\times T_G).
\end{equation}
Obviously, this implies that $\overline{W}_\Gamma$ is rational polyhedral since both $W$ and $\sigma_\Gamma^{-1}(S_\Gamma\times T_G)$ are. It remains to justify
\eqref{GSRING2E215}.

By \eqref{GSRING2E211}, $\sigma_\Gamma(D)\in S_\Gamma\times T_G$ for all $D\in \overline{W}_\Gamma$.

On the other hand, for every $D\in W\cap \sigma_\Gamma^{-1}(S_\Gamma\times T_G)$, we have
\begin{equation*}\label{GSRING2E216}
D = P + N = F_1 + F_2
\end{equation*}
where $D = P + N$ is the Zariski decomposition of $D$ with $\supp(N) \subset G$,
$F_1$ is $\BR$-effective with $\supp(F_1) \subset \Gamma$, $F_2 E = 0$ and
$F_2 C \ge 0$ for all irreducible components $E\subset \Gamma$ and $C\subset G$.
Let us write
\begin{equation*}\label{GSRING2E217}
N - F_1 = A - B
\end{equation*}
for $\BR$-effective divisors $A$ and $B$ such that $\supp(A)$ and $\supp(B)$ have no common components.
Then $\supp(A)\subset \supp(N)$, $\supp(B) \subset \Gamma$ and
\begin{equation*}\label{GSRING2E218}
P + A = F_2 + B.
\end{equation*}
The components of $\supp(N)$ and $\Gamma$ have negative definite self-intersection matrices, respectively.
Therefore, $A^2 \le 0$, $B^2 \le 0$ and the equalities hold if and only if $A =0$ and $B=0$, respectively.

Since $P$ is nef and $\supp(A)$ and $\supp(B)$ have no common components,
$PB\ge 0$, $AB\ge 0$ and hence
\begin{equation*}\label{GSRING2E219}
0 \le (P + A)B = (F_2 + B) B = B^2.
\end{equation*}
Therefore, $B = 0$ and
\begin{equation*}\label{GSRING2E220}
P + A = F_2.
\end{equation*}
Since $F_2 C \ge 0$ for all components $C\subset G$ and
$\supp(A)\subset \supp(N)\subset G$,
$F_2 A \ge 0$ and hence
\begin{equation*}\label{GSRING2E221}
0 \le F_2 A = (P + A) A = A^2.
\end{equation*}
Therefore, $A = 0$. In conclusion, $P = F_2$, $N = F_1$ and $D=P+N\in \overline{W}_\Gamma$ by \eqref{GSRING2E211}. This proves \eqref{GSRING2E215}.
\end{proof}

Finally, let us finish the proof of Theorem \ref{GSRING2THM000}.
We write
$\varphi(a_1,a_2,...,a_l) = \sum a_i D_i$ for $a_i\in \BR$
and let $V_\Gamma = \varphi^{-1} (W_\Gamma)$. Then
\begin{equation*}\label{GSRING2E035}
M_{X,D,D_1,D_2,...,D_l}(t) = h^0(D) + \sum_{\Gamma\le G}
\ \sum_{\mathbf{v}\in V_\Gamma\cap \BZ^l} h^0(D+\varphi(\mathbf{v})) t^{\mathbf v}
\end{equation*}
where we write $\mathbf{v} = (m_1,m_2,...,m_l)$
and $t^{\mathbf{v}} = t_1^{m_1}t_2^{m_2}...t_l^{m_l}$.
So it suffices to show that
\begin{equation}\label{GSRING2E036}
\sum_{\mathbf{v}\in V_\Gamma\cap \BZ^l} h^0(D+\varphi(\mathbf{v})) t^{\mathbf v}
\end{equation}
is rational for all $\Gamma \le G$. For $V_\Gamma$ with empty interior, we can prove that \eqref{GSRING2E036} is rational by induction on $l$. If $V_\Gamma$ has nonempty interior,
\begin{equation}\label{GSRING2E037}
\begin{aligned}
\sum_{\mathbf{v}\in V_\Gamma\cap \BZ^l} h^0(D+\varphi(\mathbf{v})) t^{\mathbf v} &= 
\sum_{\mathbf{v}\in \overline{V}_\Gamma\cap \BZ^l} h^0(D+\varphi(\mathbf{v})) t^{\mathbf v}
\\
&\quad - \sum_{\mathbf{v}\in (\overline{V}_\Gamma - V_\Gamma)\cap \BZ^l} h^0(D+\varphi(\mathbf{v})) t^{\mathbf v}
\end{aligned}
\end{equation}
Note that $\overline{V}_\Gamma - V_\Gamma$ is a finite union of cones with empty interior. So the second term on the right hand side of \eqref{GSRING2E037} is rational by induction hypothesis. It comes down to showing the rationality of
\begin{equation}\label{GSRING2E038}
\sum_{\mathbf{v}\in \overline{V}_\Gamma\cap \BZ^l} h^0(D+\varphi(\mathbf{v})) t^{\mathbf v}.
\end{equation}
By Lemma \ref{GSRING2LEMZARISKI}, we can show that
for every pair of $\BQ$-divisors $D_1, D_2\in \overline{W}_\Gamma$,
$P_1 N_2 = P_2 N_1 = 0$ under the Zariski decomposition
$D_i = P_i+N_i$.

As before, by dividing $\overline{V}_\Gamma$ into a union of
rational simplicial cones, we can reduce the rationality of \eqref{GSRING2E038} to that of
\begin{equation*}\label{GSRING2E039}
\sum_{\mathbf{v}\in S\cap \BZ^l} h^0(D+\varphi(\mathbf{v})) t^{\mathbf v}
\end{equation*}
for all rational simplicial cones $S\subset \overline{V}_\Gamma$.

Suppose that $S$ is generated $\mathbf{v}_1, 
\mathbf{v}_2, ..., \mathbf{v}_l\in \BZ^l$ over $\BR$. Then
\begin{equation*}\label{GSRING2E040}
\Big(\sum \widehat{P}_i\Big) \Big(\sum \widehat{N}_i\Big) = 0
\end{equation*}
for the Zariski decompositions $\varphi(\mathbf{v}_i) = \widehat{P}_i + \widehat{N}_i$.

We have
\begin{equation*}\label{GSRING2E041}
\begin{aligned}
&\quad \sum_{\mathbf{v}\in S\cap \BZ^l} h^0(D+\phi(\mathbf{v})) t^{\mathbf v}\\
&= \sum_{\mathbf{u} \in \Lambda} t^{\mathbf u} \sum_{m_i\in \BN}
h^0((D + \varphi(\mathbf{u})) + \sum m_i \varphi(\mathbf v_i))
t^{\sum m_i \mathbf v_i}
\end{aligned}
\end{equation*}
where
\begin{equation*}\label{GSRING2E043}
\Lambda = \Big\{\sum_{i=1}^l \lambda_i \mathbf v_i: 0\le \lambda_i < 1\Big\}\cap \BZ^n.
\end{equation*}
This reduces it to the case \eqref{GSRING2E025}.

\appendix

\section{Examples of $E_X$}\label{GSRING2APPENDIXA}

The Euler-Chow series is in general very hard to compute. There are however a few examples where it is rational. Here we write two known examples, for details see \cite{ElizondoEulerSeries}. In there, it is proved that the invariant subvarieties under the torus action are in bijection with the cones of the Fan associated to the toric varieties. Let $\CC_{\lambda}$ be the Chow variety of effective cyles with homology class equal to $\lambda \in H_{2p}(X,\mathbb{Z})$, and $\CC_{\lambda}^T$ the subset of fixed points under the torus action. It is well known that $\chi(\CC_{\lambda})=\chi(\CC_{\lambda}^{T}).$ 

\subsection{Projective space $\Pp^n$}

Let  \, $X = {\bf P}^{n}$ \, be the complex projective space of dimension $n$. 
Let $\{e_{1},\ldots,e_{n}\}$ be the standard basis for ${\bf R}^{n}$. Consider $A \, = \, \{e_{1},\ldots,e_{n+1}\}$ a set of generators of the
fan $\Delta$, where $e_{n+1} = - \sum_{i=1}^{n} e_{i}$. We have the following equality 
$$
  H^{\ast}(X, \, {\bf Z}) \, \cong  \, {\bf Z} 
   \, [t_{1}, \ldots, t_{n+1}] \, /I
$$
where $I$ is the ideal generated by
$$
  i) \; \; \;   t_{1} \cdots t_{n+1}
$$ 
and 
$$
ii) \; \; \; t_i \sim t_j 
$$
Therefore
$$
H^{\ast} \, (X, \, {\bf Z}) \, = \, {\bf  Z} \, [t] \, / t^{n+1}.
$$
Consequently, any two cones of dimension $n-p$ represent the same element in
cohomology, and  
$$
{\displaystyle \prod_{i=1}^{(_{n-p}^{n+1})} {\left( \frac{1}{1-t} \right)}
\, = \, 
\left( \frac{1}{1-t} \right)^{(_{n-p}^{n+1})} \, = \, 
\left( \frac{1}{1-t} \right)^{(_{p+1}^{n+1})} \, = \, E_{p}(X)}.
$$
\subsection{Hirzebruch surfaces}
A set of generators for the fan $\Delta$
that represents the Hirzebruch surface $X(\Delta)$ is
given by $\{e_{1}, \ldots , e_{4} \}$ with
$\{e_{1}, e_{2}\}$ the standard basis for ${\bf R}^{2}$, and
$e_{3} \, = \, -e_{1} + ae_{2}, \; \; a   1$
and $e_{4} \, = \, -e_{2}$.
With the same notation as in the last examples, we have
$$
H^{\ast} (X(\Delta)) \, = \, {\bf Z}[t_{1}, \ldots, t_{4}] \, / \, I
$$
where $I$ is generated by
$$
i) \; \; \; \{ t_{1}t_{3}, \, t_{2}t_{4} \}
$$
and 
$$
ii) \; \; \; \{ t_{1} -t_{3}, \, t_{2}+ at_{3} -t_{4} \}
$$
from $ii)$ we have the following conditions for the $t_{i}$'s in $H^{\ast}(X)$
\begin{equation}
\label{coh}
\; \; \; \; t_{1} \, \sim \, t_{3} \; \; \mbox{and }\; \; t_{2}
\, \sim \, (t_{4}-at_{3}).
\end{equation}
A basis for $H^{\ast} (X)$ is given by $\{ \{0\}, t_{3}, t_{4}, t_{4}t_{1}  \}$
\, (see \cite{Toric}). 
The Euler series for each dimension is:

\noindent  {\bf 1)} Dimension 0:
There are four orbits (four cones of dimension 2), and all
of them are equivalent in homology.  
$$
E_{0} \, = \, {\left(\frac{1}{1-t}\right)}^{4}
$$
{\bf 2)}  Dimension 1:
Again, there are four orbits (four cones of dimension 1), and the
relation  among them,  
in homology, is given by ~\ref{coh}. We obtain 
$$
E_{1} \, = \, {\left(\frac{1}{1-t_{3}}\right)}^{2} \,
\left(\frac{1}{1-t_{4}}\right) \, \left(\frac{1}{1-t_{3}^{-a} t_{4}}\right).
$$
{\bf 3)}  Dimension 2: The only orbit is the torus itself so
$$
E_{2} \, = \, \frac{1}{1-t}.
$$

\newcommand{\etalchar}[1]{$^{#1}$}

\end{document}